\documentclass[12pt,nopagetitle]{article}
\usepackage{amsfonts, amsmath, amssymb, amsthm}  
\usepackage[english]{babel}
\usepackage[utf8]{inputenc}
\usepackage[OT1]{fontenc}
\usepackage{bm}

\usepackage{textcomp, cmap, comment}	
\usepackage{graphicx, wrapfig}
\usepackage{epigraph, xcolor, hyperref}
\usepackage{array,tabularx,tabulary,booktabs}
\usepackage{euscript, mathrsfs}	

\newcounter{iii}

\newcommand{\bb}{{\mathcal B}}
\newcommand{\aaa}{{\mathcal A}}
\newcommand{\G}{{\mathcal G}}
\newcommand{\hh}{{\mathcal H}}

\newcommand{\ff}{\mathcal F}
\theoremstyle{plain}
\newtheorem{thm}{Theorem}

\newtheorem{lem}[thm]{Lemma}

\theoremstyle{definition}

\newtheorem{prop}[thm]{Proposition}

\title{Diversity of uniform intersecting families}
\author{Andrey Kupavskii\footnote{Moscow Institute of Physics and Technology, University of Birmingham;  Email: {\tt kupavskii@yandex.ru} \ \ Research supported by the grant RNF~16-11-10014.}}
\date{}

\begin{document}
\maketitle
\begin{abstract} A family $\ff\subset 2^{[n]}$ is called {\it intersecting}, if any two of its sets intersect. Given an intersecting family, its {\it diversity} is the number of sets not passing through the most popular element of the ground set. Peter Frankl made the following conjecture: for $n> 3k>0$ any intersecting family $\ff\subset {[n]\choose k}$ has diversity at most ${n-3\choose k-2}$. This is tight for the following ``two out of three'' family: $\{F\in {[n]\choose k}: |F\cap [3]|\ge 2\}$.  In this note we prove this conjecture for $n\ge ck$, where $c$ is a constant independent of $n$ and $k$. In the last section, we discuss the case $2k<n<3k$ and show that one natural generalization of Frankl's conjecture does not hold.
\end{abstract}

\section{Introduction}
We denote $[n]:=\{1,\ldots, n\}$, $2^{[n]}:=\{S:S\subset [n]\}$ and ${[n]\choose k}:=\{S:S\subset [n], |S|=k\}$. Any subset of $2^{[n]}$ we call a {\it family}. A family $\ff\subset 2^{[n]}$ is called {\it intersecting}, if any two of its sets intersect. The {\it degree} $\delta_i$ of an element $i\in[n]$ is the number of sets from $\ff$ containing $i$. We denote by $\Delta(\ff)$ the largest degree of an element: the maximum of $\delta_i$ over $i\in [n]$. The {\it diversity} $\gamma(\ff)$ of $\ff$ is the number of sets, not containing the element of the largest degree: $\gamma(\ff):=|\ff|-\Delta(\ff)$.

The study of intersecting families started from the famous Erd\H os-Ko-Rado theorem \cite{EKR}, and since then a lot of effort was put into understanding the structure of large intersecting families. The EKR theorem states that the largest uniform intersecting family consists of all sets containing a given element, that is, the maximal family of diversity 0. The Hilton-Milner theorem \cite{HM} gives the largest size of the family with diversity at least 1. Frankl's theorem \cite{Fra1}, especially in its strengthened version due to Kupavskii and Zakharov \cite{KZ} bounds the size of the families with diversity at least ${n-u-1\choose n-k-1}$, where $3\le u\le k$ is a fixed real number. We also refer to \cite{Kup}, where, among other results, a conclusive version of this theorem was obtained.

\begin{thm}[\cite{KZ}]\label{thmkz} Let $n>2k>0$ and $\ff\subset {[n]\choose k}$ be an intersecting family. Then, if $\gamma(\ff)\ge {n-u-1\choose n-k-1}$ for some real $3\le u\le k$, then \begin{equation}\label{eq01}|\ff|\le {n-1\choose k-1}+{n-u-1\choose n-k-1}-{n-u-1\choose k-1}.\end{equation}
\end{thm}

It is easy to see that the theorem above is sharp for each integer $u\in [3,k]$: consider the families
$$\mathcal A_u:=\{F\in{[n]\choose k}: F\supset [2,u+1] \text{ or } 1\in F, F\cap [2,u+1]\ne \emptyset\}, \ \ \ \ \ \ \ u\in [2,k].$$
The family $\mathcal A_3$ has diversity ${n-4\choose k-3}$ and size ${n-1\choose k-1}+{n-4\choose k-3}-{n-4\choose k-1} = 3{n-3\choose k-2}+{n-3\choose k-3}$. The family $\mathcal A_2$ has the same size as $\mathcal A_3$ (and this is why the case $u=2$ does not appear in the Theorem~\ref{thmkz}), but the diversity of $\mathcal A_2$ is bigger: it is equal to ${n-3\choose k-2}$.

The following problem was suggested by Katona and addressed by Lemons and Palmer \cite{LP}: what is the maximum diversity of an intersecting family $\ff\subset {[n]\choose k}$? They found out that for $n>6k^3$ we have $\gamma(\ff)\le {n-3\choose k-2}$, with the equality possible only for $\mathcal A_2$ and some of its subfamilies. Recently, Frankl \cite{Fra6} (Theorem 2.4) proved that $\gamma(\ff)\le {n-3\choose k-2}$ for all $n\ge 6k^2$, and conjectured that the same holds for $n>3k$.

The purpose of this note is to prove the following theorem
\begin{thm}\label{thm1} There exists a constant $C$, such that for any $n>Ck>0$ any intersecting family $\ff\subset {[n]\choose k}$ satisfies $\gamma(\ff)\le {n-3\choose k-2}$. Moreover, if $\gamma(\ff)={n-3\choose k-2}$, then $\ff$ is a subfamily of an isomorphic copy of $\aaa_2$.
\end{thm}

We note that a somewhat similar proof strategy, which first uses results on Boolean functions to obtain some rough structure for the problem, and then uses combinatorics to obtain a precise result, was recently used by Keller and Lifshitz \cite{KL}  in a much more general setting.
\section{Proof of Theorem~\ref{thm1}}

The following theorem, proven by Dinur and Friedgut \cite{DF}, is the main ingredient in the proof. We say that a family $\mathcal J\subset 2^{[n]}$ is a {\it $j$-junta}, if there exists a subset $J\subset [n]$ of size $j$ (the {\it center} of the junta), such that the membership of a set in $\ff$ is determined only by its intersection with $J$, that is, for some family $\mathcal J^*\subset 2^{J}$ (the {\it defining family}) we have $\ff=\{F:F\cap J\in \mathcal J^*\}$.

\begin{thm}[\cite{DF}]\label{thmdf} For any integer $r\ge 2$,there exist functions $j(r), c(r)$, such that for any integers $1 < j(r) < k < n/2$, if $\ff\subset {[n]\choose k}$ is an intersecting family with $|\ff|\ge c(r){n-r\choose k-r}$, then there exists an intersecting $j$-junta $\mathcal J$ with
$j\le j(r)$ and \begin{equation}\label{eqDF} |\ff\setminus\mathcal J|\le c(r){n-r\choose k-r}.\end{equation}
\end{thm}

We start the proof of the theorem.  Choose $C$ sufficiently large (its choice will become clear later), $n>Ck>0$ and an intersecting family $\ff\subset{[n]\choose k}$. Then, applying Theorem~\ref{thmdf} with $r=5$, we get that there exists a $j$-junta $\mathcal J$, $j\le j(5)$, such that $|\ff\setminus\mathcal J|\le c(5){n-5\choose k-5}<{n-5\choose k-4}$, where the second inequality holds provided $C$ is large enough.

The first step is to show that, unless $\mathcal J = \mathcal A_2$, we have $\gamma(\ff)<{n-3\choose k-2}$.

\begin{prop} Consider an intersecting $j$-junta $\mathcal J\subset 2^{[n]}$, with center $J\subset [n], |J|=j,$ and defined by an intersecting family $\mathcal J^*\subset 2^J$. Then $\mathcal J$ satisfies one of the two following properties:
\begin{itemize}
  \item $\mathcal J$ is contained in a family isomorphic to $\mathcal A_2$.
  \item There exists  $i\in J$, such that all sets from $\mathcal J^*$ of size at most 2 contain $i$.
\end{itemize}
\end{prop}
\begin{proof} Note that the intersecting families of $\le 2$-element sets which cannot be pierced by a single element are isomorphic to  ${[3]\choose 2}$. Therefore, the junta that does not fall into the second category must have the center of size $3$ and  be defined by a family containing all 2-element subsets of the center.  Then we are only left to observe the fact that $\mathcal A_2$ is a junta with center $J=[3]$ and defined by the family $\mathcal J^*={[3]\choose 2}\cup [3]$:

$$\aaa_2=\big\{A\in {[n]\choose k}: |A\cap [3]|\ge 2\big\}.$$
\end{proof}

Assume that $\mathcal J$ is not isomorphic to $\mathcal A_2$. Then, as it follows from the proposition above, $\gamma(\mathcal J)\le 2^{j}{n-j\choose k-3}$. If $C=n/k$ is sufficiently large, then $$2^j{n-j\choose k-3}\le \frac{2^j}{C}{n\choose k-2}<\frac{2^{j+1}}{C}{n-3\choose k-2}\le \frac{1}{2}{n-3\choose k-2}.$$
Moreover, ${n-5\choose k-4}< \frac 12{n-3\choose k-2}$ for any $n\ge 2k$. Therefore, in this case we can conclude that $$\gamma(\ff)\le \gamma(\mathcal J)+|\ff\setminus \mathcal J|<{n-3\choose k-2}.$$

From now on we suppose that $\mathcal J=\mathcal A_2$. For $i=1,2,3$ consider the families $\ff_i:=\{F\in \ff: F\cap [3]=\{i\}\}$. W.l.o.g., assume that $\ff_1$ has the largest size among $\ff_i$. We will use the following obvious bound: $\gamma(\ff)\le |\ff|-\delta_1$. Consider the following three families on $[4,n]$:
\begin{align*}\G:=&\big\{F\cap [4,n]: F\in \ff, F\cap [3]=\{2,3\}\big\},\\
\hh_1:=&\big\{F\cap [4,n]: F\in\ff_1\big\},\\
\hh_2:=&\big\{F: F\in\ff, F\subset [4,n]\big\}.
\end{align*}
 Clearly, $\G\subset {[4,n]\choose k-2}$, $\hh_1\subset{[4,n]\choose k-1}$, $\hh_2\subset {[4,n]\choose k}$.  Most importantly, $$\gamma(\ff)\le |\G|+|\ff_2|+|\ff_3|+|\hh_2| \le |\G|+2|\ff_1|+|\hh_2| = |\G|+2|\hh_1|+|\hh_2|.$$
Therefore, to conclude the proof of the theorem, it is sufficient to show the following two inequalities:
\begin{align}\label{eq05}|\mathcal G|+4|\hh_1|&\le {n-3\choose k-2},\\ \label{eq06}
|\mathcal G|+2|\hh_2|&\le {n-3\choose k-2}.\end{align}
Summing these two inequalities with coefficients equal to $1/2$, we get that $\gamma(\ff)\le {n-3\choose k-2}$.

There are two important properties that we are going to use. The first one is that $\ff\setminus \mathcal J=\ff_1\cup\ff_2\cup\ff_3\cup \hh_2$, and thus, using $|\hh_1|=|\ff_1|$, we have $|\hh_1|, |\hh_2|\le |\mathcal F\setminus \mathcal J|\le {n-5\choose k-4}$. The second one is that the pair of families $\G,\hh_1$ as well as $\G, \hh_2$ are {\it cross-intersecting}. We say that two families are {\it cross-intesecting}, if any set from one intersects any set from the other.

In what follows we show that \eqref{eq05}, \eqref{eq06} hold in a more general form. Similar inequalities appeared in \cite{KZ}, \cite{FK1} and \cite{Kup}.

\begin{lem}\label{lemkey} Consider a set $[m]$ and two cross-intersecting families $\mathcal A\subset{[m]\choose a}, \mathcal B\subset {[m]\choose b}$. Assume that $m>(C'+1)\cdot\max\{a, b\}$ for some constant $C'$. Assume also that $|\mathcal B|\le {m-(b-a+1)\choose a-1}$. Then \begin{equation}\label{eq666} |\mathcal A|+C'|\mathcal B|\le {m\choose a}.\end{equation}
\end{lem}

Before proving the lemma, let us deduce the inequalities \eqref{eq05}, \eqref{eq06} out of \eqref{eq666} and thus conclude the proof of Theorem~\ref{thm1}. For \eqref{eq05} we need to substitute $\aaa:=\G,\ \bb:=\hh_1$,\ $a:=k-2,\ b:=k,\ C':=C,\ [m]:=[4,n]$. Then we conclude that \eqref{eq05} holds even with 4 replaced by  $C$. The deduction of \eqref{eq06} is similar. Moreover, we get that a pair of families may achieve equality in \eqref{eq05} and \eqref{eq06} only if $\hh_1=\hh_2=\emptyset$ (and therefore $\ff_i=\emptyset$ for $i\in[3]$). Therefore, if $\gamma(\ff)={n-3\choose k-2},$ then $\ff\subset \aaa_2$. \\

To prove Lemma~\ref{lemkey}, we need to give some definitions, related to the famous Kruskal-Katona theorem. A {\it lexicographical order} (lex) on the sets from ${[n]\choose k}$ is an order, in which $A$ is less than $B$ iff  the minimal element of $A\setminus B$ is less than the minimal element of $B\setminus A$.
 For $0\le m\le {n\choose k}$ let $\mathcal L(m,k)$ be the collection of $m$ largest sets with respect to lex.

\begin{thm}[\cite{Kr},\cite{Ka}]\label{thmHil}Suppose that $\mathcal A\subset {[n]\choose a}, \mathcal B\subset {[n]\choose b}$ are cross-intersecting. Then the families $\mathcal L(|\mathcal A|,a),\mathcal L(|\mathcal B|, b)$ are also cross-intersecting.
\end{thm}

\begin{proof}[Proof of Lemma~\ref{lemkey}]
Using Theorem~\ref{thmHil}, we may w.l.o.g. assume that $\aaa = \mathcal L(|\mathcal A|,a), \bb = \mathcal L(|\mathcal B|, b)$. Due to the restriction on the size of $\bb$, any set in it contains $[b-a+1]$. Consider the families
\begin{align*}\mathcal B_0:=&\{B\setminus [b-a+1]:B \in \bb\},\\ \mathcal A_0:=&\{A: A\in \aaa, A\cap [b-a+1]=\emptyset \}.\end{align*}
 Put $Y:=[b-a+2,m]$ (if $b<a$, put $Y:=[1,m]$). Note that $|Y| = \min\{m-(b-a+1), m\}$.  Clearly, $\bb_0\subset{Y\choose a-1}$ and $\aaa_0\subset {Y\choose a}$. Consider a bipartite graph $G$ with parts ${Y\choose a},\ {Y\choose a-1}$, and edges connecting disjoint sets. Then the intersection of $\aaa\cup \bb$ with the parts of the graph is $\aaa_0\cup\bb_0$, and it forms an independent set in $G_0$. Thus, we have
 $$\frac{|\aaa_0|}{{|Y|\choose a}}+\frac{|\bb_0|}{{|Y|\choose a-1}}\le 1.$$
We have ${|Y|\choose a}/{|Y|\choose a-1} = (|Y|-a)/a = \min\{m-a, m-b-1\}/a\ge C'$.  This implies
$$|\aaa_0|+C'|\bb_0|\le |\aaa_0|+\frac{{|Y|\choose a}|\bb_0|}{{|Y|\choose a-1}}\le {|Y|\choose a}.$$
The lemma follows from the fact that
$$|\aaa|+ C'|\bb|\le {m\choose a}-{|Y|\choose a}+|\aaa_0|+C'|\bb_0|\le {m\choose a}.$$
\end{proof}




\section{What happens when $2k\le n\le 3k$?}

Under the same assumption that we make in Theorem~\ref{thm1}, it is possible to prove certain Hilton-Milner type stability results for diversity (using more elaborate versions of Lemma~\ref{lemkey}). However, we think that it is more interesting to resolve the problem for any  $n> 3k$ and show that the family with the maximum possible diversity must be isomorphic to a subfamily of $\aaa_2$, or the ``two out of three'' family. When $2k<n< 3k$, then other families have larger diversity. They can be described as ``$r+1$ out of $2r+1$'' families:
\begin{equation}\mathcal D_r:=\big\{D\in {[n]\choose k}: |D\cap [2r+1]|\ge r+1 \big\}, \ \ \ \ \ \ \ \ r=1,\ldots, k-1.\end{equation}

The following seems to be a reasonable conjecture at a first glance.
\vskip+0.1cm

\textbf{Conjecture } {\it Fix $n\ge 2k> 0$ and consider an intersecting family $\ff\subset {[n]\choose k}$.  If for some $r\in \mathbb Z_{\ge 0}$ we have
$(k-1)\big(2+\frac 1{r+1}\big)+1\le n\le (k-1)\big(2+\frac 1r\big)+1$, then $\gamma(\ff)\le \gamma(\mathcal D_r)$.}\vskip+0.1cm

Substituting $r=0$ in the conjecture, we get that $\gamma(\ff)\le {n-3\choose k-2}$ for any $n\ge 3k-2$. Let us explain what stands behind this naive conjecture. Assume that the  element with the highest degree in $\ff$ is $1$. Then the conjecture is just stating that, if one restricts the attention to the family $\ff':=\{F\in \ff: 1\notin F\},\ \ff\subset {[2,n]\choose k}$, then the size of $\ff'$ is at most the size of the largest 2-intersecting family on $[2,n]$. We say that a family is {\it $t$-intersecting}, if any two sets from the family intersect in at least $t$ elements. The exact formulas given in the naive conjecture come from the famous Complete Intersection Theorem by Ahlswede and Khachatrian \cite{AK}.

The families $\mathcal D'_r\subset{[2,n]\choose k}, \mathcal D'_r:=\{D\in \mathcal D_r: 1\notin D\}$ are 2-intersecting. And it comes as no surprise. Indeed, the same must be true for any {\it shifted} intersecting family $\ff$. Let us first give the definition of shifting. \\

For a given pair of indices $1\le i<j\le n$ and a set $A \subset [n]$ define its {\it $(i,j)$-shift} $S_{i,j}(A)$ as follows. If $i\in A$ or $j\notin A$, then $S_{i,j}(A) = A$. If $j\in A, i\notin A$, then $S_{i,j}(A) := (A-\{j\})\cup \{i\}$. That is, $S_{i,j}(A)$ is obtained from $A$  by replacing $j$ with $i$.

The  $(i,j)$-shift $S_{i,j}(\mathcal F)$ of a family $\mathcal F$ is as follows:

$$S_{i,j}(\mathcal F) := \{S_{i,j}(A): A\in \mathcal F\}\cup \{A: A,S_{i,j}(A)\in \mathcal F\}.$$

We call a family $\mathcal F$ \textit{shifted}, if $S_{i,j}(\mathcal F) = \mathcal F$ for all $1\le i<j\le n$.\\

For any shifted family $\delta_1(\ff)=\Delta(\ff)$ and, if $\ff$ is intersecting, then $\ff'$ must be 2-intersecting. Indeed, if there are two sets $F_1,F_2\in \ff$, such that $F_1,F_2\subset [2,n]$ and $F_1\cap F_2 = \{x\}$, then, by shiftedness, $F'_1:=F_1\setminus \{x\} \cup \{1\}$ also belongs to $\ff$, and we have $F_1'\cap F_2 = \emptyset$, a contradiction. Consequently, the naive conjecture is true for such $\ff$).

Therefore, the conjecture above states that {\it any} intersecting family should behave as shifted intersecting families with respect to diversity.
Shifting preserves the property of a family to be intersecting, but, unfortunately, it does not allow to control the diversity of a family. This is why the general case cannot be directly reduced to shifted case. In fact, it cannot be reduced to the shifted case at all: the conjecture is false for families that are not shifted!

\subsection{Intersecting families with the largest diversity are not shifted}
Here we present a counterexample to the conjecture above found and communicated to us by Noam Lifshitz. As the counterexample shows, at least in some cases the extremal value of $\gamma(\ff)$ is attained on the families that are not shifted, which is unexpected for a problem concerning intersecting families.\\

We use the notions and results coming from the analysis of Boolean functions. We give all the necessary definitions, and all the standard results used here may be found in \cite{OD}. For a real number $0<p<1$ and a set $F\subset [n]$ and $\ff\subset 2^{[n]}$, define the $p$-biased measure $\mu_p(F):=p^{|F|}(1-p)^{n-|F|}$ and $\mu_p(\ff):=\sum_{F\in\ff}\mu_p(F)$. The {\it influence} $I_i^p(\ff)$ of coordinate $i$ in $\ff$ is $$I_i^p(\ff):=\mu_p\big(\{F: |\{F, F\Delta\{i\}|\cap \ff=1\}\big),$$ and the {\it total influence} is $I^p(\ff):=\sum_i I_i^p(\ff)$. In case if $\ff$ is closed upwards, we have \begin{equation}\label{eqinfmonotone} I_i^p(\ff) = p^{-1}\mu_p\big(\{F\in \ff:i\in F\}\big)- (1-p)^{-1}\mu_p\big(\{F\in \ff:i\notin F\}\big).\end{equation}

Fix a sufficiently large $r$ and even bigger $k\ge k_0(r)$, $n\ge n_0(r)$, satisfying the conditions on $n$ from the conjecture. Put $p:=\frac{k}{n}$. That is, $p =\frac 12 - (1+o(1))\frac 1r$.\\

\textbf{Intersecting family with low influences. } In what follows, we describe the family $\mathcal T_r$, which restriction $\mathcal T_r\cap J$ (see below)  provides an example showing that the Kahn-Kalai-Linial inequality~\cite{KKL} is sharp for indicator functions of intersecting families. The family $\mathcal T_r$ has larger diversity than  $\mathcal D_r$. The example is taken from Gil Kalai's post on MathOverflow \cite{GKM}, however, since the explanation of the necessary properties in the post was very brief, we expand the exposition here, hopefully providing all the necessary details.

Consider an intersecting family $\mathcal T_r\subset {[n]\choose k}$, which is a $(2r+1)$-junta with center  $J:=[2r+1]$, and $\mathcal T_r\cap J$ is the following intersecting family. Arrange the elements of $J$ on the circle, and for each set $S\subset 2^{J}$ form a sequence $\mathbf{u}:=(u_1,u_2,\ldots)$, where $u_i$ is the length of the $i$-th longest run of consecutive $1$'s; similarly, $\mathbf{z}:=(z_1,z_2,\ldots )$ is the sequence, in which $z_i$ is the $i$-th longest run of consecutive $0$'s.  Form $\mathcal T_r\cap J$ by including all sets, for which its sequence $\mathbf{u}$ is lexicographically bigger than $\mathbf{z}$ (we denote it $\mathbf{u}\succ\mathbf{z}$). Note that, since $|J|$ is odd, we cannot have equality between the sequences. Therefore, we have $$|\mathcal T_r\cap J|=2^{|J|-1}$$ since if $T\subset J$ is in $\mathcal T_r\cap J$, then its complement $J\setminus T$ is not, and vice versa.

Let us show that $\mathcal T_r\cap J$ is an intersecting family. Assume the contrary, and let $T_1,T_2$ be two disjoint sets in $\mathcal T_r\cap J$. Let $\mathbf{u}^i$, $\mathbf{z}^i$, $i=1,2$, be the corresponding one and zero runs sequences. Then, clearly, $\mathbf{z}^1\succ\mathbf{u}^2$ and $\mathbf{z}^2\succ\mathbf{u}^1$, otherwise, it would be impossible to fit the runs of $1$'s of $T_1$ inside the runs of $0$'s of $T_2$ (and the same with the roles of $T_1$, $T_2$ interchanged). However, if, say, $\mathbf{u}^1\succ\mathbf{u}^2$, then by transitivity $\mathbf{z}^2\succ\mathbf{u}^2$, a contradiction.

In what follows, all logarithms have base $2$ and all asymptotic notations are with respect to $r\to \infty$.

\begin{lem} \label{leminf} For each $i\in J$, we have $I_i^{1/2}(\mathcal T_r\cap J)= O\big(\frac{\log r}{r}\big)$.
\end{lem}
\begin{proof}The proof requires a somewhat tedious analysis of the typical sets in the family. The family is clearly transitive on $J$, and thus it is sufficient to show that $I^{1/2}(\mathcal T_r\cap J)=O(\log r)$. The total influence $I^{1/2}(\mathcal T_r\cap J)$ is the average number of {\it pivotal coordinates} in a randomly chosen set from $J$ according to $\mu_{1/2}$, that is, the number of coordinates which change results in the set passing from $\mathcal T_r\cap J$ to its complement or vice versa.

Choose a random set $T\in  J$ according to $\mu_{1/2}$ and denote its zero and one runs sequences $\mathbf{z}:=(z_1,z_2,\ldots )$ and $\mathbf{u}:=(u_1,u_2,\ldots )$, respectively. The coordinate is pivotal, if after its change the lexicographical order of $\mathbf z$ and $\mathbf u$ is reversed.

%
Using first moment, it is easy to see that with probability $1-o(1/r)$ the largest run of consecutive $1$'s in $T$ has size at most  $(2+o(1))\log r$, and the same for the runs of zeros. Thus, the sequences not satisfying this property contribute $o(1)$ to the total influence. In what follows we ignore such sequences.


In what follows, we assume that $\mathbf u\succ \mathbf z$. The other case is treated analogously.
Choose $\rho$ such that $z_{\rho}=u_{\rho}$ for $j\le \rho$ and $u_{\rho+1}>z_{\rho+1}$. The lex order is reversed if at least one of the two happens: either $(u_1,\ldots,u_{\rho+1})$ is replaced by a lexicographically smaller sequence, or $(z_1,\ldots, z_{\rho+1})$ is replaced by a lexicographically larger sequence. Denote the number of the former and latter types by $s_1$ and $s_2$, respectively.
The number of pivotal coordinates is at most $s_1+s_2$, but we will bound just $s_1$ instead. A pivotal coordinate $k$ of the second type is a pivotal coordinate of the first type for the set $T\setminus\{k\}$, and, since $\mu_{1/2}(T) = \mu_{1/2}(T\setminus\{k\})$, the average value of $s_1$ is not more than twice smaller than $s_1+s_2$, Thus if $\mathrm E[s_1]=O(\log r)$, then $\mathrm E[s_1+s_2] = O(\log r)$.



From the above, we clearly have $s_1\le \sum_{j=1}^{\rho+1} u_{\rho}$. Moreover, $u_i=O(\log r)$ for each $i=1\ldots, \rho+1$. Consequently, we can bound $\mathrm E[s_1]\le \mathrm E[\sum_{j=1}^{\rho+1}O(\log r)]= O(\log r)\sum_{k=1}^{\infty}\Pr[\rho\ge k]$. We use a slight  variation of this bound. Since the number of pivotal coordinates is at most $2r+1$, we can bound \begin{equation}\label{eqevil}\mathrm{E} [s_1]\le O(\log r)\sum_{k=1}^{r^{0.1}}\Pr[\rho\ge k] +(2r+1)\Pr[\rho\ge r^{0.1}].\end{equation}
To complete the proof that $\mathrm{E} [s_1]=O(\log r)$ (and thus the proof of the lemma), it is  sufficient to show the validity of the following lemma
\begin{lem}\label{lem10}
For any $\rho\le r^{0.1}$ we have $\Pr[\rho\ge k]\le e^{-\alpha \log^2k}+o(1/r)$ for some $\alpha<1$. In particular, $\Pr[\rho\ge r^{\beta}]=o(1/r)$ for any fixed $\beta>0$.
\end{lem}
We defer the proof of this lemma to the end of the section, and finish the proof of Lemma~\label{leminf} modulo Lemma~\ref{lem10}. From it, we have $\Pr[\rho\ge r^{0.1}]=o(1/r)$, and the right hand side of the inequality \eqref{eqevil} is at most $O(\log r)\sum_{k=1}^{\infty} e^{-\alpha \log^2 k}+o(1) = O(\log r).$
\end{proof}

Consider the family $\mathcal T^{\uparrow}_r:=\{T\subset [n]:T\cap J\in \mathcal T_r\cap J\}$. Then we have $\mu_{1/2}(\mathcal T^{\uparrow}_r)=1/2$. Similarly, define the family $\mathcal D^{\uparrow}_r$ based on $\mathcal D_r$. Again, $\mu_{1/2}(\mathcal D^{\uparrow}_r)=1/2$. Define the {\it $p$-biased diversity} of $\ff$ as $\gamma_p(\ff):=\min_{i\in [n]}\mu_p(\{F\in \ff:i\notin F\})$.

Using the result of Dinur and Safra \cite{DS}, for sufficiently large $n=n(r)$ we have \begin{align}
\label{eqmeasureapprox2} \Big|\gamma_p(\mathcal T^{\uparrow}_r)-\frac{\gamma(\mathcal T_r)}{{n\choose k}}\Big|\le & \frac 1{r^2},
\end{align}
and the same for $\mathcal D_r$ and $\mathcal D^{\uparrow}_r$. (Note that we use the fact that both families $\mathcal T_r\cap\{T\subset J: 1\notin T\}$ and $\mathcal D_r\cap\{T\subset J: 1\notin T\}$ are $2r+1$-juntas.)
Using the Margulis-Russo lemma, for any upwards closed family $\ff\subset 2^{J}$ and $p_0\in(0,1)$, we have
$$\frac{d \mu_p(\ff)}{d p}|_{p_0}=I^{p_0}(\ff).$$
Using large deviation estimates, it is not difficult to see, that for any $p_0\in [p,1/2]$ the contribution of sets from $\ff$ of size not in $[r-r^{2/3}, r+r^{2/3}]$ to the influence is negligible (since the measure of such sets is negligible). On the other hand, for any $F$ of size in  $[r-r^{2/3}, r+r^{2/3}]$, we have $\mu_{p_0}(F)=(1+o(1))\mu_{1/2}(F)$ for any $p_0\in [p,1/2]$. Thus, $I^{p_0}(\ff)=(1+o(1))I^{1/2}(\ff)$ for any such $p_0$. We conclude that we have $$\mu_{1/2}(\ff)-\mu_{p}(\ff)=(1+o(1))(1/2-p)I^p(\ff).$$
At the same time, for a symmetric, closed upward  family $\ff\subset 2^{J}$ and for any $i\in J$ we have $pI_i^p(\ff)+\frac 1{1-p}\gamma_p(\ff) = \mu_p(\ff).$ Therefore,
$$\gamma_p(\ff)= (1-p)\big(\mu_p(\ff)-\frac{p}{|J|}I^p(\ff)\big) = (1-p)\mu_{1/2}(\ff)-(1-p+o(1))\frac{p+\big(\frac 12-p\big)|J|}{|J|}I^p(\ff).$$

We know that $I^p(\mathcal D_r\cap J)=\Omega(\sqrt r)$, since $\mathcal D_r\cap J$ is the majority function. On the other hand, using Lemma~\ref{leminf}, we have $I^p(\mathcal T_r\cap J)=O(\log r).$ Using the displayed formula above, we get
$$\gamma_p(\mathcal D^{\uparrow}_r)=\gamma_p(\mathcal D_r\cap J)=\frac {1-p}2-\Omega\big(\frac{1}{\sqrt r}\big) \ \ \ \ \text{and} \ \ \ \ \gamma_p(\mathcal T^{\uparrow}_r)=\gamma_p(\mathcal T_r\cap J)=\frac {1-p}2-O\big(\frac{\log r}{r}\big).$$
Combining the formulas above with \eqref{eqmeasureapprox2}, we conclude that $\gamma(\mathcal D_r)<\gamma(\mathcal T_r)$. \\

\begin{proof}[Proof of Lemma~\ref{lem10}]

How to express the condition that for a random sequences its vectors of zero and one runs share the first $k$ coordinates? Let $N(t)$ be a random variable counting the number of runs of length at least $t$ in the sequence. Choose an integer $t_0$ such that $N(t_0)\le k$, but $N(t_0-1)>k$. In order to have $\rho\ge k$, exactly a half of each $N(t)$, $t\ge t_0$, must be zero runs, and a half must be one runs. In what follows, we analyze the behaviour of $N(t)$.

Take an integer $t$ and fix the value of $N(t)$. For each run of length at least $t$, reveal the values of the first $t$ coordinates that belong to the run (in the clockwise order), as well as the value that precedes the run clockwise. E.g., for $t=3$ the sequence may look like $xx0111xxx1000xx1000x\ldots$, where $x$ stands for the coordinates that are not revealed. Let us denote $L_j$, $j\in[N(t)],$ the intervals of unrevealed coordinates between the revealed coordinates. If some $L_j$ has length smaller than $r^{1/10}$, then  reveal its coordinates, otherwise  keep it intact. Let us denote $\mathcal S$ the class of all possible subsequences that can be fixed (revealed) in this way. Each subsequence $S\in \mathcal S$ gives rise to a family $\mathcal C(S)$ of sequences containing $S$ as a subsequence.  Fix any subclass $\mathcal S'\subset \mathcal S$. Then for a randomly chosen cyclic sequence $R$ we have
\begin{equation}\label{fullprob}\Pr[\rho\ge k]\le \sum_{S\in \mathcal S\setminus \mathcal S'}\Pr[\rho\ge k|R\in \mathcal C(S)]\Pr[R\in \mathcal C(S)]+\sum_{S\in\mathcal S'}\Pr[R\in \mathcal C(S)].\end{equation}
That is, one may think of $\mathcal S'$ as a small set of ``exceptional'' classes. We will show that the latter term on the right hand side is $o(1/r)$, while $\Pr[\rho\ge k|R\in \mathcal C(S)]\le e^{-\alpha \log^2 k}$ for each $S\in \mathcal S\setminus \mathcal S'$. Since $\sum_{s\in  \mathcal S\setminus \mathcal S'}\Pr[R\in \mathcal C(S)]\le 1$, this will conclude the proof of the lemma.

  We note that the expected value of $N(t)$ is $r2^{-t}$ (we choose the starting point, choose arbitrarily the coordinate $x$ before the starting point, and then fix the $t$ coordinates to be equal to $1-x$). Moreover, using the Talagrand inequality (e.g., in the form of \cite[Theorem 7.7.1]{AS}), we can show that when the expectation is, say, bigger than $r^{1/10}$, then the value of $N(t)$ is well-concentrated around the expectation (it is equal to $(1+o(1))\mathrm E[N(t)]$ with probability at least $1-r^{-c}$ for any $c>0$). We restrict our attention only on the values of $t$ such that the expected value of $N(t)$ does not exceed $r^{1/10}$, and we include in $\mathcal S'$ all sequences for which the value of $N(t)$ exceeds $2r^{1/10}$. By the above, there are $o(1/r)$ of those.

  Fix some $t$ satisfying the condition $\mathrm E[N(t)]\le r^{1/10}$ and note that $t=\Omega(\log r)$.  The probability that there is more than one interval $L_j$ of unrevealed coordinates of length $\le r^{1/10}$, given that $N(t)\le 2r^{1/10}$, is $o(1/r)$. Indeed,  compare the number of all possible choices for the starting positions  of the runs of length $\ge t$ (roughly ${r\choose N(t)}$) with the number of choices, in which we first fix $N(t)-2$ starting positions of such runs, and then choose the remaining two at distance at most $2r^{1/10}$ from one of the already chosen ones (at most ${r\choose N(t)-2}\cdot (4 r^{1/5})^2$). Include all such subsequences in $\mathcal S'$. We are not going to include any more subsequences in $\mathcal S'$ and we note that there are $o(1/r)$ sequences that contain subsequences from $\mathcal S'$.


  Fix one subsequence $S\in \mathcal S\setminus \mathcal S'$. We aim to bound $\Pr[\rho\ge k|R\in \mathcal C(S)]$. Consider a uniform distribution over all sequences in $\mathcal C(S)$. Recall that at least $[N(t)]-1$ of $L_j$ are unrevealed. We note that the only restriction on the choices of coordinates in $L_j$ is that they cannot contain runs of ones or zeros of length $\ge t$, moreover, there is no dependency between the choices of coordinates in different $L_j$.  Next, for each $j$ we reveal the first coordinate $x_j$ of $L_j$ in the clockwise order. We claim that $\Pr[x_j=1]=(\frac 12+o(1))$. Indeed, for each admissible sequence starting from $x_j$ we change $x_j$ to $1-x_j$ and obtain an admissible sequence and vice versa, unless the $t-1$ elements immediately following $x_j$ have the same value, while $x_j$ had the opposite value. But this constitutes at most a $1/2^{t}$-fraction of all possible admissible sequences on $L_j$, and thus only affects the value of $\Pr[x_j=1]$ by $o(1)$.

  Since the choices for different $x_j$ are independent, we have the following. First, the expected number $\mathrm E[N(t+1)]$ of ``surviving'' runs of length $t+1$ is $(\frac 12+o(1))N(t)$. Moreover, it is tightly concentrated around the expectation: using a Chernoff-type bound, we have $\Pr[|N(t+1)-\frac 12N(t)|\ge \frac 16 N(t)]\le e^{-cN(t)}$ for some fixed positive constant $c$.

  Now we are in position to bound $\Pr[\rho\ge k|R\in \mathcal C(S)]$. First, we find $t_0$ as in the first paragraph of the proof of the lemma. By the previous paragraph, with probability at least $1-e^{-ck}$ for some fixed $c>0$ we have $3N(t_0)\ge k$ and $27 N(t_0+2)\ge k$. In order for $\rho\ge k$ to hold, a half of runs contributing to $N(t_0)$ should be one runs, and the other half should be zero runs. Moreover, the same should be true for $N(t)$ for $t\ge t_0+1$. Thus, we can obtain the following rough bound:\begin{small}
  $$\Pr[\rho\ge k|R\in \mathcal C(S)]\le \Pr[\rho\ge N(t_0+2)|R\in \mathcal C(S)]\cdot P_1\le e^{-ck}+\Pr\big[\rho\ge \frac k{27}|R\in \mathcal C(S)\big]\cdot P_1,$$
  \end{small}
  where $P_1$ is the probability that the values $x_j$ will be chosen in such a way that the number of zero and one runs of length at least $t_0+1$ is the same. It is easy to see that
  $$P_1\le (1+o(1))2^{-N(t_0)}\sum_{j=0}^{N(t_0)/2}{N(t_0)/2\choose j}^2=\Theta\big(N(t_0)^{-1/2}\big).$$
  (note here that we have to take into account the interval $L_j$ that was fixed  and that potentially gave one zero or one run of length $t_0$, but it does not affect the validity of the bound above). Looking at the recursion $\Pr[\rho\ge k]\le e^{-ck}+\Theta(k^{-1/2})\Pr[\rho\ge k/27]$, it is not difficult to conclude that $\Pr[\rho\ge k]\le e^{-\alpha\log^2k}$ for some positive constant $\alpha$. Substituting into \eqref{fullprob}, we get the result.
  \end{proof}

We remark that the problem treated in Lemma~\ref{lem10} seems to be interesting on its own and that it would be desirable to have a fuller understanding of the behaviour of the probability $\Pr[\rho\ge k]$. However, we believe that the bound on the probability (modulo $o(1/r)$ and the constant $\alpha$ in the exponent) is essentially sharp.

\textbf{Remark. } While preparing the second version of the manuscript, Hao Huang provided another counterexample to the conjecture in the range $3k\le n \le (2+\sqrt 3) k$ \cite{HH}. We presented our counterexample because we thought that the method used to derive it is interesting in its own right.

\textsc{Acknowledgements:} I thank Peter Frankl for introducing me to the concept of diversity and the problem studied in the paper, as well as for many fruitful discussions and interesting ideas that he shared with me. I thank Noam Lifshitz for proposing a simplified proof of Lemma~\ref{lemkey} and especially for finding the beautiful counterexample from the last section, as well as pointing out several errors in the earlier version of the text of Section~3.1. I also thank Maksim Zhukovskii for helpful discussions on the proof of Lemma~\ref{lem10}.

\end{document}